\documentclass[12pt,a4]{amsart}
\usepackage[a4paper, left=28mm, right=28mm, top=28mm, bottom=34mm]{geometry}
\usepackage[all]{xy}
\usepackage{comment}
\usepackage{amsmath}
\usepackage{amssymb}
\usepackage{amsthm}
\usepackage{mathrsfs}
\usepackage{mathtools}
\usepackage{mathabx}
\theoremstyle{definition}
\newtheorem{thm}{Theorem}[section]
\newtheorem{lem}[thm]{Lemma}
\newtheorem{cor}[thm]{Corollary}
\newtheorem{prop}[thm]{Proposition}
\newtheorem{conj}[thm]{Conjecture}
\theoremstyle{definition}

\newtheorem{rem}[thm]{Remark}

\newtheorem{defn}[thm]{Definition}

\numberwithin{equation}{section}
\def\A{{\mathbb A}}
\def\F{{\mathbb F}}
\def\Q{{\mathbb Q}}
\def\R{{\mathbb R}}
\def\Z{{\mathbb Z}}
\def\C{{\mathbb C}}

\def\Ker{\mathop{\mathrm{Ker}}\nolimits}

\def\Pic{\mathop{\mathrm{Pic}}\nolimits}

\def\GL{\mathop{\mathrm{GL}}\nolimits}
\def\SL{\mathop{\mathrm{SL}}}
\def\SU{\mathop{\mathrm{SU}}\nolimits}
\def\Her{\mathop{\mathrm{Her}}\nolimits}

\def\Tr{\mathop{\mathrm{Tr}}\nolimits}
\def\det{\mathop{\mathrm{det}}\nolimits}
\def\dim{\mathop{\mathrm{dim}}\nolimits}
\def\div{\mathop{\mathrm{div}}\nolimits}

\def\Ind{\mathop{\mathrm{Ind}}\nolimits}

\def\L{\mathscr{L}}
\def\H{\mathscr{H}}

\def\M{\mathscr{M}}
\def\F{\mathscr{F}}
\def\W{\mathscr{W}}
\def\OO{\mathscr{O}}
\def\Res{\mathop{\mathrm{Res}}\nolimits}
\def\GSpin{\mathop{\mathrm{GSpin}}\nolimits}
\def\CH{\mathop{\mathrm{CH}}\nolimits}

\def\exp{\mathop{\mathrm{exp}}\nolimits}
\def\Sp{\mathop{\mathrm{Sp}}\nolimits}
\def\SO{\mathop{\mathrm{SO}}\nolimits}
\def\SU{\mathop{\mathrm{SU}}}

\def\Mp{\mathop{\mathrm{Mp}}\nolimits}
\def\Supp{\mathop{\mathrm{Supp}}\nolimits}

\def\U{\mathrm{U}}

\def\S{\textbf{\textsl{S}}}

\newcommand{\transp}[1]{{}^{t}\!{#1}}
\newcommand{\defeq}{\vcentcolon=}

\begin{document}

\title[modularity on unitary Shimura varieties]
{Modularity of special cycles on unitary Shimura varieties over CM-fields}

\author{Yota Maeda}
\address{Department of Mathematics, Faculty of Science, Kyoto University, Kyoto 606-8502, Japan}
\email{y.maeda@math.kyoto-u.ac.jp}

%\date{\today}
\date{October 3, 2019}


% 11F03 Modular and automorphic functions
% 11F70 Representation-theoretic methods; automorphic representations over local and global fields
% 11F80 Galois representations
% 11R39 Langlands-Weil conjectures, nonabelian class field theory
% 14G25 Global ground fields
% 14F30 $p$-adic cohomology, crystalline cohomology

\date{\today}

\maketitle

\begin{abstract}
  We study the modularity of the generating series of special cycles on unitary Shimura varieties over CM-fields  of degree $2d$ associated with a Hermitian form in $n+1$ variables whose signature is $(n,1)$ at $e$ real places and $(n+1,0)$ at the remaining $d-e$ real places for $1\leq e <d$.
  For $e=1$, Liu proved the modularity and Xia showed the absolute convergence of the generating series.
  On the other hand, Bruinier constructed regularized theta lifts on orthogonal groups over totally real fields and proved the modularity of special divisors on orthogonal Shimura varieties.
  By using Bruinier's result, we work on the problem for $e=1$ and give another proof of Liu's one.
  For $e>1$, we prove that the generating series of special cycles of codimension $er$ in the Chow group is a Hermitian modular form of weight $n+1$ and genus $r$,  assuming the Beilinson-Bloch conjecture with respect to orthogonal Shimura varieties.
  Our result is a generalization of $\textit{Kudla's modularity conjecture}$, solved by Liu unconditionally when $e=1$.
\end{abstract}

\section{Introduction}
Hirzebruch-Zagier \cite{HZ} observed that the intersection number of special divisors on Hilbert modular surfaces generates a certain weight 2 elliptic modular form.
Kudla-Millson generalized this work in \cite{KM} and they proved that special cycles on orthogonal (resp. unitary) Shimura varieties generate Siegel (resp. Hermitian) modular forms with coefficients in the cohomology group.
Recently, Yuan-Zhang-Zhang \cite{YZZ} treats this problem in the Chow group in the case of orthogonal Shimura varieties and proved the modularity.
Kudla \cite{Remarks} and the author \cite{Yota} generalized this problem for a certain orthogonal Shimura variety under the Beilinson-Bloch conjecture.

In this paper, we shall work on unitary case in the Chow group.
Our problem is Conjecture \ref{Conjecture}.
We give two solutions to this problem (Corollary \ref{MainCor1} and Theorem \ref{MainTheorem3}).
First, we prove Conjecture \ref{Conjecture} for $e=1$  unconditionally by using Bruinier's result \cite{Bruinier}.
On the other hand, for $e=1$, Liu \cite{Liu} solved Conjecture \ref{Conjecture},\ i.e. proved the modularity of special cycles on unitary Shimura varieties in the Chow group, assuming the absolute convergence of the generating series.
Recently, Xia showed the modularity and absolute convergence of the generating series for $e=1$. 
Our result in this paper gives another proof of Liu's result.
For $e=1$ and $r=1$, the modularity of special divisors are proved in Theorem \ref{MainTheorem2}.
To treat higher codimensional cycles, we adopt the induction method \cite{YZZ}.
Second, for $e>1$, we show the Conjecture \ref{Conjecture} under the Beilinson-Bloch conjecture with respect to orthogonal Shimura varieties.
we reduce the problem to the orthogonal case (\cite{Remarks} and \cite{Yota}) so that we also need the Beilinson-Bloch conjecture with respect to orthogonal Shimura varieties.
We remark that we do not prove the absolute convergence of the generating series in this paper.

Before giving the statement of our results, we shall define some Shimura varieties.
\subsection{Unitary Shimura varieties}
\label{Subsection:Unitary Shimura varieties}
Let $d$, $e$ and $n$ be positive integers such that $e<d$. %($d\neq 1$?)
Let $F$ be a totally real field of degree $d$ with real embeddings $\sigma_1,\dots,\sigma_d$ and $E$ be a CM extension of $F$.
We write $\partial_F$ for the different ideal of $F$.
Let $(V_E,\langle\ ,\ \rangle)$ be a non-degenerate Hermitian space of dimension $n+1$ over $E$ whose signature is $(n,1)$ at $\sigma_1,\dots,\sigma_e$ and $(n+1,0)$ at $\sigma_{e+1},\dots,\sigma_d$.

For $i=1,\dots,e$, let  $V_{E,\sigma_i,\C}\defeq V_E\otimes_{F,\sigma_i}\C$. %and $\mathbb{P}(V_{E,\sigma_i,\C})\defeq(V_{E,\sigma_i,\C}\backslash \{0\})/\C^{\times}$.
and $D^E_i \subset \mathbb{P}(V_{E,\sigma_i,\C})$ be the Hermitian symmetric domain defined as follows:
\[D^E_i\defeq\{v \in V_{E,\sigma_i,\C}\backslash \{0\} \mid \langle v,v \rangle >0\}/\C^\times .\]
We put
\[D_E\defeq D^E_1\times\dots\times D^E_e.\]
Let $\U(V_E)$ be the unitary group of $V_E$ over $F$, which is a reductive group over $F$.
We put $G\defeq\Res_{F/\Q}\U(V_E)$ and consider the Shimura varieties associated with the Shimura datum $(G,D_E)$.
Then, for any open compact subgroup $K_f^G\subset G(\A_f)$, the Shimura datum $(G,D_E)$ gives a Shimura variety $M_{K_f^G}$ over $\C$, whose $\C$-valued points are given as follows:
\[M_{K_f^G}(\C)=G(\Q) \backslash (D_E \times G(\A_f))/K_f^G.\]
Here $\A_f$ is the ring of finite ad\`{e}les of $\Q$.
We remark that $M_{K_f^G}$ has a canonical model over a number field called the reflex field.
Hence $M_{K_f^G}$ is canonically defined over $\overline{\Q}$.
$\overline{\Q}$ denotes an algebraic closure of $\Q$ embedded in $\C$.
By abuse of notation, in this paper, the canonical model of $M_{K_f^G}$ over $\overline{\Q}$ is also denoted by the same symbol $M_{K_f^G}$.
Then the Shimura variety $M_{K_f^G}$ is a projective variety over $\overline{\Q}$ since $0<d-e$.
It is a smooth variety over $\overline{\Q}$ if $K_f^G$ is sufficiently small.
In this paper, we assume that $K_f^G$ is sufficiently small.

\subsection{Orthogonal Shimura varieties}
\label{Subsection:Orthogonal Shimura varieties}

We define $V_F\defeq V_E$ considered as $E$-vector space and $(\ ,\ )\defeq \Tr_{E/F}\langle\ ,\ \rangle$.
Then $(V_F,(\ ,\ ))$ is a quadratic space of dimension $2n+2$ over $F$ whose signature is $(2n,2)$ at $\sigma_1,\dots,\sigma_e$ and $(2n+2,0)$ at $\sigma_{e+1},\dots,\sigma_d$.
We define $D_F$ similarly.
We put $H\defeq \Res_{F/\Q}\GSpin(V_F)$ and define $N_{K_f^H}$ similarly for an open compact subgroup $K_f^H\subset H(\A_f)$.
Let $L\subset V_F$ be a lattice and $L'$ denotes the dual lattice.
Now we get a group embedding, $G\hookrightarrow H$.
From here, we assume $K_f^G=H(\A_f)\cap K_f^H$ so that
\begin{eqnarray}
  \label{embedding}
\iota\colon M_{K_f^G}\hookrightarrow N_{K_f^H}.
\end{eqnarray}
In this paper, we also assume $K_f^H$ is sufficiently small.
\begin{comment}
There exists a line bundle $\F_k$ whose sections are weight $k$ modular forms on $N_{K_f}^H$.
We remark that $\F_1=H(\Q)\backslash \L\times H(\A_f)/K_f^H$.
\end{comment}

\subsection{Special cycles on Shimura varieties}
\label{Subsection:Special cycles on Shimura varieties}
We shall define the special cycles on unitary Shimura varieties.
For $i=1,\dots,e$, let $\L_i\in \Pic(D_i^E)$ be the line bundle which is the restriction of $\mathscr{O}_{\mathbb{P}(V_{E_i,\sigma_i,\C})}(-1)$ to $D_i^E$.
By pulling back to $D_E$, we get $p^*_i\L_i\in \Pic(D_E)$,
where $p_i\colon D_E \to D_i^E$ are the projection maps.
These line bundles descend to
$\L_{K_f,i} \in \Pic(M_{K_f^G})\otimes_{\Z}\Q$ and thus we obtain
$\L\defeq\L_{K_f^G,1}\otimes\dots\otimes\L_{K_f^G,e}$ on $M_{K_f^G}$.

We shall define special cycles following Kudla \cite{OrthogonalKudla}, \cite{Remarks}.
Let $W\subset V_E$ be a totally positive subspace over $E$.
We denote $G_W\defeq\Res_{F/\Q}\U(W^\perp)$.
Let $D_{W,E}\defeq D_{W,1}^E\times\dots\times D_{W,e}^E$ be the Hermitian
symmetric domain associated with $G_W$, where
\[D_{W,i}^E\defeq\{w \in D_i^E \mid \forall v \in W_{\sigma_i}, \ \langle v,w \rangle=0\} \quad (1\leq i\leq e).\]
Then we have an embedding of Shimura data $(G_W,D_{W,E})\hookrightarrow (G,D_E)$.
For any open compact subgroup $K_f^G \subset G(\A_f)$ and $g \in G(\A_f)$,
we have an associated Shimura variety $M_{gK_f^Gg^{-1},W}$ over $\C$:
\[M_{gK_f^Gg^{-1},W}(\C)=G_W(\Q) \backslash (D_{W,E} \times G_W(\A_f))/(gK_f^Gg^{-1} \cap G_W(\A_f)).\]
Assume that $K_f^G$ is neat so that the following morphism
\begin{align*}
M_{gK_f^Gg^{-1},W}(\C) &\to M_{K_f^G}(\C)\\
[\tau,h]&\mapsto [\tau,hg]
\end{align*}
is a closed embedding \cite[Lemma 4.3]{Remarks}.
Let $Z^G(W,g)_{K_f^G}$ be the image of this morphism.
We consider $Z^G(W,g)_{K_f^G}$ as an algebraic cycle of codimension $e\dim_FW$ on $M_{K_f^G}$ defined over $\overline{\Q}$.

For any positive integer $r$ and $x=(x_1,\dots,x_r) \in V_E^r$, let $U(x)$ be the $E$-subspace of $V_E$ spanned by $x_1,\dots,x_r$.
We define the \textit{special cycle} in the Chow group
\[Z^G(x,g)_{K_f^G} \in \CH^{er}(M_{K_f^G})_{\C}\defeq\CH^{er}(M_{K_f^G})\otimes_{\Z}\C\]
by
\[Z^G(x,g)_{K_f^G}\defeq Z^G(U(x),g)_{K_f^G}(\mathrm{c}_1(\L^{\vee}_{K_f^G,1})\cdots\mathrm{c}_1(\L^{\vee}_{K_f^G,e}))^{r-\dim U(x)}\]
if $U(x)$ is totally positive.
Otherwise, we put $Z^G(x,g)_{K_f^G}\defeq 0$.

For a Bruhat-Schwartz function $\phi_f\in \S(V_E(\A_f)^r)^{K_f^G}$, \textit{Kudla's generating function} is defined to be the following formal power series with coefficients in $\CH^{er}(M_{K_f^G})_{\C}$ in the variable $\tau=(\tau_1,\dots,\tau_d)\in (\H_r)^d$:
\[Z_{\phi_f}^G(\tau)\defeq\sum_{x\in G(\Q)\backslash V_E^r}\sum_{g\in G_x(\A_f)\backslash G(\A_f)/K_f^G}\phi_f(g^{-1}x)Z^G(x,g)_{K_f^G}q^{T(x)}.\]
Here $G_x\subset G$ is the stabilizer of $x$, $\H_r$ is the Siegel upper half plane of genus $r$, $T(x)$ is the moment matrix $\frac{1}{2}((x_i,x_j))_{i,j}$, and
\[q^{T(x)}\defeq\exp(2\pi \sqrt{-1} \sum_{i=1}^d \Tr{\tau_i\sigma_iT(x)}).\]

For a $\C$-linear map $\ell\colon\CH^{er}(M_{K_f^G})_{\C}\to \C$, we put
\[\ell(Z_{\phi_f}^G)(\tau)\defeq\sum_{x\in G(\Q)\backslash V_E^r}\sum_{g\in G_x(\A_f)\backslash G(\A_f)/K_f^G}\phi_f(g^{-1}x)\ell(Z^G(x,g)_{K_f^G})q^{T(x)},\]
which is a formal power series with complex coefficients in the variable $\tau\in(\H_r)^d$.
We define $Z_{\phi_f}^H(\tau)$ similarly.

\begin{rem}
  We explain $Z^H_{\phi_f}(\tau)$ is an analogy of a theta function.
  For a totally real definite matrix $\beta\in M_r(F)$, let $\Omega_{\beta}\defeq\{x\in V_F^r\mid T(x)=\beta\}$ and we consider the Fourier expansion with respect to $\beta$.
  Now we choose $\beta$ such that $\Omega_{\beta}\neq \varnothing$ and fix $x_0\in\Omega_{\beta}(\A_f)$.
  For $\xi_j\in H(\A_f)$,
  \[\Supp(\phi_f)\cap\Omega_{\beta}(\A_f)=\coprod_{j=1}^{\ell}K_f^H\cdot\xi_j\cdot x_0,\]
  and we put
  \[Z^H(\beta,\phi_f)_{K_f^H}\defeq\sum_{j=1}^{\ell}\phi_f(\xi_j^{-1}\cdot x_0)Z^H(x_0,\xi_j)_{K_f^H}.\]
  Then $Z_{\phi_f}^H(\tau)$ becomes
  \[Z^H_{\phi_f}(\tau)=\sum_{\beta\geq0}Z^H(\beta,\phi_f)_{K_f^H}q^{\beta}\]
  and by adding Kudla-Millson forms and Gaussian functions, this is exactly a theta function in the cohomology group.
  For details, see \cite{OrthogonalKudla}.
\end{rem}

Before stating our goal, we have to clarify the notion of ``modular".

\begin{defn}
\label{def:modular}
  Let $V$ be a vector space over $\C$ and $f$ be a formal power series with coefficients in $V$.
  We say $f$ is modular if for any $\C$-linear map $\ell\colon V\to\C$ such that $\ell(f)$ is absolutely convergent, $\ell(f)$ is modular.
\end{defn}

\subsection{The Beilinson-Bloch conjecture}
To state the main theorem for the $e>1$ case, we have to introduce the Beilinson-Bloch conjecture because we need the modularity for the orthogonal case and this was proved under the Beilinson-Bloch conjecture in \cite{Remarks} or \cite{Yota}.

Let $X$ be a smooth variety over $\overline{\Q}$ and 
\[cl^m:\CH^m(X)\to H^{2m}(X,\Q)\defeq H^{2m}(X(\C),\Q)\]
be the $m$-th cycle map.
On the other hand, we have the $m$-th intermediate Jacobian $J^{2m-1}(X)$ of $X$ defined by the Hodge structure of $X$ (see before \cite[Conjecture 1.2]{Yota}).
Then there exists the $m$-th higher Abel-Jacobi map:
\[AJ^m:\Ker(cl^m)\otimes\Q\to J^{2m-1}(X)\otimes\Q.\]
The Beilinson-Bloch conjecture claims $AJ^m$ is injective.
Hence if $H^{2m-1}(X,\Q)=0$, then under the Beilinson-Bloch conjecture for $m$, the map
\[cl^m_{\Q}:\CH^m(X)\otimes\Q\to H^{2m}(X,\Q)\defeq H^{2m}(X(\C),\Q)\]
is injective.
See \cite{Yota} for the detailed claim of the Beilinson-Bloch conjecture.

\subsection{Main results}
\label{Subsection:Main results}
For notations, see subsections \ref{Subsection:Orthogonal Shimura varieties} and \ref{Subsection:Special cycles on Shimura varieties}.
In the context of \textit{Kudla's modularity conjecture}, our problem is as follows.
\begin{conj}
\label{Conjecture}
The generating series $Z_{\phi_f}^G(\tau)$ is a Hermitian modular form of weight $n+1$ and genus $r$.
\end{conj}
We give two partial solutions to this problem in this paper.
See Corollary \ref{MainCor1} and Theorem \ref{MainTheorem3}.

\begin{comment}
\begin{thm}
\label{MainTheorem1}
Assume $e=1$.
Let $f$ be a weakly holomorphic Whittaker form of weight $k$ whose coefficients $c(m,\mu)$ are integral.
Then there exists a meromorphic modular form $\Psi_f(\tau,g)$ for $G(\Q)$ of level $K_f^G$ satisfying
\begin{enumerate}
\item The weight of $\Psi$ is $-B(f)$.

\item $\div\Psi=Z(f)$
\end{enumerate}
\end{thm}
\end{comment}
\begin{comment}
\begin{rem}
  Theorem \ref{MainTheorem1} reduces to \cite[Theorem 6.8]{Bruinier} by using an embedding (\ref{embedding}).
  Hofmann constructed Borcherds products on unitary groups on imaginary quadratic field \cite{Hofmann}.
  Theorem \ref{MainTheorem1} generalizes \cite[Theorem 8.1]{Hofmann} to CM-field.
\end{rem}
\end{comment}
First, we can prove the modularity of the  generating series of special divisors by using the regularized theta lift on orthogonal groups.

\begin{thm}
\label{MainTheorem2}
Assume $e=1$ and $r=1$.
Then $Z_{\phi_f}^G(\tau)$ is an elliptic modular form of weight $n+1$.
\end{thm}
Theorem \ref{MainTheorem2} generalizes \cite[Theorem 10.1]{Hofmann}.
We can prove stronger result by induction on $r$ \cite{YZZ}.
See Corollary \ref{MainCor1}.
There exists an exceptional isomorphism $\SL_2\cong\SU(1,1)\subset\U(1,1)$, so it doesn't follow immediate from \cite{Hofmann} or Theorem \ref{MainTheorem2} that $Z_{\phi_f}^G(\tau)$ is a Hermitian modular form, i.e. 
Theorem \ref{MainTheorem2} shows only $\SU(1,1)$-modularity of $Z_{\phi_f}^G(\tau)$.
However, we can show the $\U(1,1)$-modularity of  $Z_{\phi_f}^G(\tau)$ by proving term-wise modularity.
This means that we can show the modularity of $Z_{\phi_f}^G(\tau)$ for the parabolic subgroup $P_1$ and a specific element $w_1$ defined in Section \ref{Section:Modularity}.
On the other hand, $P_1$ and $w_1$ generate $\U(1,1)$ and we already know the modularity for $w_1\in\SU(1,1)$ by Theorem \ref{MainTheorem2}, so the problem reduces to proving the modularity for $P_1$. 
For the proof of the modularity for the parabolic subgroup $P_1$, see \cite{Liu}, \cite{Yota} and \cite{YZZ}.
By combining the above modularity and induction on $r$, we can prove the modularity of special cycles of higher codimension.

\begin{cor}
\label{MainCor1}
Assume $e=1$.
Then $Z_{\phi_f}^G(\tau)$ is a Hermitian modular form of weight $n+1$ and genus $r$.
\end{cor}

  This gives another proof of Theorem \ref{MainTheorem3} for the $e=1$ case and \cite[Theorem 3.5]{Liu}.
  This is shown unconditionally differently from Theorem \ref{MainTheorem3}.

Now we state the theorem for $e>1$.
We remember that $G\defeq\Res_{F/\Q}\U(V_E)$ is the unitary group associated with a Hermitian space $V_E$ over a CM field $E$ and for a Bruhat-Schwartz function $\phi_f\in \S(V_E(\A_f)^r)^{K_f^G}$, our generating series $Z_{\phi_f}^G(\tau)$ is defined as follows with coefficients in $\CH^{er}(M_{K_f^G})_{\C}$ in the variable $\tau=(\tau_1,\dots,\tau_d)\in (\H_r)^d$:
\[Z_{\phi_f}^G(\tau)\defeq\sum_{x\in G(\Q)\backslash V_E^r}\sum_{g\in G_x(\A_f)\backslash G(\A_f)/K_f^G}\phi_f(g^{-1}x)Z^G(x,g)_{K_f^G}q^{T(x)}.\]
Our main result in this paper is as follows.
\begin{thm}
\label{MainTheorem3}
$Z_{\phi_f}^G(\tau)$ is a Hermitian modular form of weight $n+1$ and genus $r$ under the Beilinson-Bloch conjecture for $m=e$ with respect to orthogonal Shimura varieties.
\end{thm}
\begin{rem}
  We assume the Beilinson-Bloch conjecture for $m=e$ for $N_{K_f^H}$ when $2n\geq3$, i.e. $n>1$.
 When $n=1$, we need to assume the Beilinson-Bloch conjecture for $m=e$ for a larger orthogonal Shimura variety $N_{K_f^H}'$ including $N_{K_f^H}$, see \cite[Theorem 1.6]{Yota}.
 For the precise statement of the Beilinson-Bloch conjecture, see \cite[Subsection 1.2]{Yota}.
\end{rem}
\begin{rem}
  \label{unitaryBB}
  Kudla \cite{Remarks} and the author \cite{Yota} proved the modularity of generating series associated with orthogonal Shimura varieties for $e>1$.
  Their results are shown by using the Kudla-Millson's cohomological coefficient result \cite{KM} and reducing the problem to this cohomological case under the Beilinson-Bloch conjecture with respect to orthogonal Shimura varieties.
  So one might think that the modularity of generating series associated with unitary Shimura varieties would also be proved in the same way, but the Hodge numbers appearing in the cohomology of unitary Shimura varieties do not seem to vanish \cite[Remark 1.2]{Remarks}.

  Historically, for unitary Shimura varieties, Kudla-Millson \cite{KM} studied the cohomological coefficients case.
  In the Chow group, Hofmann \cite{Hofmann} showed the $\SL_2$-modularity of the generating series over imaginary quadratic fields for the $r=1,e=1$ case and Liu \cite{Liu} showed Hermitian modularity for $e=1$ case, assuming the absolute  convergence of the generating series.
  Therefore we give a generalization of their work.
  On the other hand, Xia \cite{Xia} showed Liu's result, not assuming the absolute   convergence of the generating series.
  He uses the formal Fourier-Jacobi series method  similar to the work over $\Q$ of Bruinier-Westerholt-Raum \cite{BW}.
  
\end{rem}
  Theorem \ref{MainTheorem2} and Corollary \ref{MainCor1} are included Theorem \ref{MainTheorem3} under the Beilinson-Bloch conjecture, but we give another proof we can apply only for $r=1$, using regularized theta lifts.

We can also restate the result using the Kudla's modularity conjecture for  orthogonal Shimura varieties as follows.

\begin{cor}
\label{restate}
$Z_{\phi_f}^G(\tau)$ is a Hermitian modular form of weight $n+1$ and genus $r$, assuming the modularity of the generating series of special cycles on orthogonal Shimura varieties for $r=1$.
\end{cor}
For the reason why we only assume the modularity for $r=1$ on orthogonal Shimura varieties, see subsection \ref{subsection:w_1,r=1}.

\begin{comment}
\subsection{Modular forms}
\label{Subsection:Modular forms}
we construct the theta lift on unitary groups, so first, define modular forms on $G(\Q)$.
Let $\phi_f(\tau,g)$ be a function on $D_E\times G(\A_f)$.
We call $\phi_f(\tau,g)$ be a meromorphic modular form of weight $k$ and level $K_f$ if

The last condition is fulfilled because $1<d$.
We put $\Phi\defeq \phi_f(\tau,g)j(,z)^{-k}$ so that we obtain a function $\Phi(\tau,g)$ on $M_K^H$ which is left $G(\Q)$ invariant and right $K_f$ invariant.
 %definition of $\L_k$
\end{comment}

\begin{comment}
\subsection{Weil representation of (O$\times$Sp)}
\label{Subsection:Weil rep of O,Sp}

\subsection{Beilinson-Bloch conjecture}
\label{Subsection:Beilinson-Bloch}
\end{comment}

\subsection{Outline of the proof of Theorem \ref{MainTheorem2} and Theorem \ref{MainTheorem3}}
As an application of the modularity of special cycles on orthogonal Shimura varieties proved by using the regularized theta lifts, we can prove Theorem  \ref{MainTheorem2} and Corollary  \ref{MainCor1}.
This is another proof of \cite{Liu} for special divisors case.
Theorem \ref{MainTheorem3} is reduced to the orthogonal case \cite{Remarks} and \cite{Yota}, so we have to assume the Beilinson-Bloch conjecture about orthogonal Shimura varieties and this is our solution to Conjecture \ref{Conjecture}.

\subsection{Outline of this paper}
In section 2, we review the modularity of the generating series of special cycles on orthogonal Shimura varieties.
In section 3, we prove the modularity for the $e=1$ case.
In section 4, we give the Hermitian modularity of special cycles for $e>1$ under the Beilinson-Bloch conjecture with respect to orthogonal Shimura varieties.
\begin{comment}
\section{Bruinier's results}
\label{Bruinier}
\subsection{The Whittaker forms}
\label{Subsection:Whittaker forms}

\end{comment}
\section{Modularity on orthogonal groups}
\label{section:Modularity on orthogonal groups}
Through in this section, let $L\subset V_F$ be an even $\OO_F$-lattice and $L'$ be the $\Z$-dual lattice of $L$ with respect to $\Tr_{F/\Q}(\ ,\ )$.
Let $\hat{\Z}\defeq \prod_{p<\infty}\Z_p$ and we define $\hat{L}\defeq L\otimes\hat{\Z}$.
We have $L'/L\cong \hat{L}'/\hat{L}$, so for $\mu\in L'/L$, let $1_{\mu}\in\S(V_F(\A_{F,f}))$ be a characteristic function associated with $\mu+\hat{L}$.

In section \ref{section:Modularity on orthogonal groups}, we assume $n>2$.
\subsection{Regularized theta lifts on orthogonal groups}
We review the results of \cite{Bruinier}.
Let
\[k\defeq(k_1,k_2,\dots,k_d)=(1-n,1+n,\dots,1+n)\in\Z^d\]
and $s_0\defeq 1-k_1=n$.
We call $k$ as weight and define the dual weight $\kappa$ to be
\[\kappa\defeq(2-k_1,k_2\dots,k_d)=(1+n,1+n,\dots,1+n)\in\Z^d.\]
We use Kummer's confluent hypergeometric function
\[M(a,b,z)\defeq \sum_{n=0}^{\infty}\frac{(a)_nz^n}{(b)_nn!}\quad(a,b,z\in\C, (a)_n\defeq \frac{\Gamma(a+n)}{\Gamma(a)}, (a)_0\defeq 1)\]
and Whittaker functions
\[M_{\nu,t}(z)\defeq e^{-z/2}z^{1/2+t}M(1/2+t-\nu,1+2t,z)\quad(t,\nu\in\C)\]
\[\M_s(v_1)\defeq |v_1|^{-k_1/2}M_{\mathrm{sgn}(v_1)k_1/2,s/2}(|v_1|) e^{-v_1/2}\quad(s\in\C,v_1\in\R)\]
%where $v_1$ is the imaginary part of $\sigma_1(\tau)$.

Now we shall define the Whittaker forms
\[f_{m,\mu}(\tau,s)\defeq C(m,k,s)\M_s(-4\pi m_1v_1)\exp(-2\pi\sqrt{-1} \Tr(m\overline{\tau})1_{\mu}) \quad(m_i\defeq\sigma_i(m))\]
where $\mu\in L'/L$ and $1_{\mu}$ is a characteristic function associated with $\mu$.

Here $C(m,k,s)$ is a normalizing factor
\[C(m,k,s)\defeq \frac{(4\pi m_2)^{k_2-1}\dots (4\pi m_d)^{k_d-1}}{\Gamma(s+1)\Gamma(k_2-1)\dots\Gamma(k_d-1)}.\]
We define

\begin{eqnarray*}
f_{m,\mu}(\tau)
&\defeq&f_{m,\mu}(\tau,s_0) \\
&=&C(m,k,s_0)\Gamma(2-k_1)(1-\frac{\Gamma(1-k_1,4\pi m_1v_1)}{\Gamma(1-k_1)})e^{4\pi m_1v_1}\exp(-2\pi\sqrt{-1}\Tr m\overline{\tau})1_{\mu}.
\end{eqnarray*}
 Now for $m\in F$,  $m>>0$ means $m_i\defeq\sigma_i(m)>0$ for all $i$ and $\partial_F$ denotes the different ideal of a totally real field $F$.
 We remember that we consider a finite $\OO_F$-module $L'/L$ equipped with a quadratic form $(\ ,\ )/2$ which takes values in $F/\partial^{-1}\OO_F$ since we assume $L$ is even. 
\begin{defn}
 A harmonic Whittaker form of weight $k$ is a function which has the form
 \[\sum_{\mu\in L'/L}\sum_{m>>0} c(m,\mu)f_{m,\mu}(\tau)\]
 for $c(m,\mu)\in\C$.
 Here the second sum runs $m\in(\mu,\mu)/2+\partial^{-1}\OO_F$.
Let $H_{k,\overline{\rho_L}}$ be the $\C$-vector space consisting of harmonic Whittaker forms of weight $k$.
\end{defn}

\begin{rem}
  Here $\rho_L$ is a lattice model of the Weil representation of the metaplectic group $\Mp_2(\hat{\OO_F})$, and $f_{m,\mu}$ satisfies a certain modularity condition on $\rho_L$ and a certain differential equation.
  For details, see \cite[Chapter 4]{Bruinier}.
\end{rem}

By our assumption on $n>2$ and $\kappa_j\geq 2$ for all $j$, there is a surjective map $\xi_k\colon  H_{k,\overline{\rho_L}}\to S_{\kappa,\rho_L}$ \cite[Proposition 4.3]{Bruinier}.
Here $S_{\kappa,\rho_L}$ is the space of Hilbert modular forms of weight $\kappa$ and type $\rho_L$.
Let $M^!_{k,\overline{\rho_L}}$ be the kernel of this map and we call elements of this space as weakly holomorphic Whittaker forms of weight $k$.

So there is an exact sequence
\[0\to M^!_{k,\overline{\rho_L}}\to H_{k,\overline{\rho_L}}\xrightarrow{\xi_k} S_{\kappa,\rho_L}\to 0.\]
This exact sequence and following pairng are an analogue of classical ones.
See Borcherds \cite{Borcherds}.
There is the Petersson inner product on $S_{k,\rho_L}$.
This pairing is non-degenerate, so induces the non-degenerate pairing between $H_{k,\overline{\rho_L}}/M^!_{k,\overline{\rho_L}}$ and $S_{k,\rho_L}$ defined by
\[\{g,f\}\defeq(g,\xi_k(f))_{\mathrm{Pet}}.\]
We recall the result \cite[Proposition 4.5]{Bruinier} that provides an explicit formula for the above  non-degenerate pairing $\{\ ,\ \}$.

\begin{prop}[{\cite[Proposition 4.5]{Bruinier}}]
\label{pairing}
For $g\in S_{\kappa,\rho_L}$ and $f\in H_{k,\overline{\rho_L}}$ with Fourier expansions 
\begin{align*}
    g&=\sum_{\nu\in L'/L}\sum_{n>>0}b(n,\nu)\exp(2\pi\sqrt{-1}\Tr(n\tau))1_{\nu} \\
    f&=\sum_{\mu\in L'/L}\sum_{m>>0} c(m,\mu)f_{m,\mu}(\tau),
\end{align*}
we have
\[\{g,f\}=\sum_{\mu\in L'/L}\sum_{m>>0}c(m,\mu)b(m,\mu).\]
\end{prop}

We remark that Whittaker forms are analog of Maas forms.
See \cite[Subsection 4.1]{Bruinier}.

For $f=\sum_{\mu}\sum_{m}c(m,\mu)f_{m,\mu}(\tau)\in H_{k,\overline{\rho_L}}$, we define $Z(f)\defeq \sum_{\mu}\sum_{m}c(m,\mu)Z^H(m,\mu)_{K_f^H}$.
Let $I\defeq\Res_{F/\Q}\SL_2$ and $\chi_V$ be a quadratic character of $\A_F^{\times}/F^{\times}$ associated to $V$ given by
\[\chi_V(x)\defeq(x,(-1)^{\ell(\ell-1)/2}\det(V))_F\quad(\ell\defeq 2n+2).\]
We review Eisenstein series \cite[subsection 6.2]{Bruinier}
Let $Q\subset H$ be the parabolic subgroup consisting of upper triangle matrices and $s\in\C$.
We take a standard section $\Phi\in I(s,\chi)\defeq \Ind_Q^H\chi_V|\cdot|^s$.
Now we have the Eisenstein series
\[E(g,s,\Phi)\defeq \sum_{\gamma\in I(F)\backslash H(F)}\Phi(\gamma g).\]
\[E(\tau,s,\ell;\Phi_f)\defeq v^{-\ell/2}E(g_{\tau},s,\Phi_f\otimes\Phi_{\infty}^{\ell})\]
where $g_{\tau}\in\Mp_2(\R)^d$ satisfies  $g_{\tau}(\sqrt{-1},\dots,\sqrt{-1})=\tau\in\H^d$ and $\Phi_{\infty}^{\ell}$ is defined in \cite[Chapter 6]{Bruinier}.
Let $1_{\mu}$ be the characteristic function of $\mu\in L'/L$.
Here the Weil representation gives an intertwining operator
\[\lambda=\lambda\otimes\lambda_f\colon\S(V(\A_F))\to I(s_0,\chi_V).\]
We obtain a vector valued Eisenstein series of weight $\ell$ with respect to $\rho_L$ by
\[E_L(\tau,s,\ell)\defeq\sum_{\mu}E(\tau,s,\ell;\lambda_f(1_{\mu}))1_{\mu}\]
We get the Fourier expansion of the Eisenstein series at $\infty$:
\[E_L(\tau,\kappa)\defeq E_L(\tau,s_0,\kappa)=1_0+\sum_{\mu\in L'/L}\sum_{m>>0}B(m,\mu)\exp(2\pi\sqrt{-1}\Tr(m\tau))1_{\mu}.\]
%We put $B(0,\mu)\defeq 1/|L'/L|$.
We define
\[B(f)\defeq \sum_{\mu\in L'/L}\sum_{m>>0}c(m,\mu)B(m,\mu)\]
for a harmonic Whittaker form $f=\sum_{\mu }\sum_{m}c(m,\mu)f_{m,\mu}$.

The following theorem is the regularized theta lift over totally real fields, proved by Bruinier  \cite[Theorem 1.3]{Bruinier}.
\begin{thm}[{\cite[Theorem 6.8]{Bruinier}}]
\label{orthoglift}
Let $f\in M!_{k,\overline{\rho_L}}$ be a weakly holomorphic Whittaker form of weight $k$ for $\Gamma=\SL_2(\OO_F)\subset I(\R)=\Res_{F/\Q}\SL_2(\R)$ whose coefficients $c(m,\mu)$ are integral.
Then there exists a meromorphic modular form $\Psi_f(\tau,g)$ for $H(\Q)$ of level $K_f^H$ satisfing
\begin{enumerate}
\item The weight of $\Psi$ is $-B(f)$.

\item $\div\Psi=Z(f)$
\end{enumerate}
\end{thm}

\subsection{Modularity of special cycles on orthogonal groups}
Now we review the modularity of special divisors on orthogonal Shimura varieties.
To state the theorem, we need to prepare the ``modified" generating series.

\[A_0(\tau)\defeq \sum_{\mu\in L'/L}-c_1(\L)1_{\mu}+\sum_{\mu\in L'/L}\sum_{m>>0}(Z^H(m,1_{\mu})_{K_f^H}+B(m,\mu)c_1(\L))q^m 1_{\mu}\]

%\[Z_0^H(m,\phi_f)\defeq Z^H(m,\phi_f)-B(m,\mu)c_1(\F_1)\]
\[A(\tau,\phi_f)\defeq -c_1(\L)+\sum_{m>>0}Z^H(m,\phi_f)_{K_f^H}q^m\]
%\[A_0(\tau,\phi_f)\defeq -c_1(\F_1)+\sum_{m}Z_0^H(m,\phi_f)q^m\]
\[A(\tau)\defeq \sum_{\mu}A(\tau,1_{\mu})1_{\mu}=\sum_{\mu\in L'/L}-c_1(\L)1_{\mu}+\sum_{\mu\in L'/L}\sum_{m>>0}Z^H(m,1_{\mu})_{K_f^H}q^m 1_{\mu}\]
%\[A_0(\tau)\defeq \sum_{\mu}A_0(\tau,1_{\mu})1_{\mu}\]

We want to show the modularity of $Z_{\phi_f}(\tau)$ but first, prove the modularity of $A_0(\tau)$.
See Remark \ref{A_0toA}.
We remark that $A(\tau,\phi_f)=Z^H_{\phi_f}(\tau)$.
The following theorem was proved by Bruinier \cite{Bruinier}.
\begin{thm}[{\cite[Theorem 7.1, Proposition 7.3]{Bruinier}}]
  \label{orthogmodular}
  For any $n>0$, 
  \[A_0(\tau)\in S_{\kappa,\rho_L}\otimes\CH^1(N_{K_f^H})\]
\end{thm}

\begin{rem}
  \label{A_0toA}
For $\C$-linear map $\ell$ such that $\ell(A_0)$ and $\ell(A)$ are absolutely convergent, we know
\[\ell(A_0(\tau))=\ell(A(\tau))+c_1(\L)E_L(\tau,\kappa)\]
by \cite[Remark 6.5]{Bruinier}.
So we also get
\[A(\tau)\in S_{\kappa,\rho_L}\otimes\CH^1(N_{K_f^H}).\]
\end{rem}

\begin{rem}
\label{rem:modularity on orthogonal groups}
  We want to show the modularity of $Z^H_{\phi_f}(\tau)$.
  Now $\phi_f$ is a locally constant, compactly supported function, so we can factorize this as $\phi_f=\sum_{\mu\in L'/L}e_{\mu}1_{\mu}$ for some $e_{\mu}\in\C$ for $\mu\in L'/L$.
  We define
  \begin{align*}
      \delta\colon S_{\kappa,\rho_L}\otimes\CH^1(N_{K_f^H})_{\C}\subset S_L\otimes\CH^1(N_{K_f^H})_{\C}&\to\CH^1(N_{K_f^H})_{\C}[[q]]\\
      \sum_{\mu\in L'/L}\sum_{m}b(m,\mu)1_{\mu}\otimes Z_{m,\mu} q^m &\mapsto \sum_{\mu\in L'/L}\sum_{m}b(m,\mu)e_{\mu}Z_{m,\mu} q^m 
  \end{align*}
  where $\sum_{\mu\in L'/L}\sum_{m}b(m,\mu)1_{\mu}q^m \in S_{\kappa,\rho_L}, Z_{m,\mu}\in\CH^1(N_{K_f^H})_{\C}$.
  Then
  $\delta(A(\tau))=Z^H_{\phi_f}(\tau)$, so this is modular by Theorem \ref{orthogmodular} and Definition \ref{def:modular}.
  See also \cite[Section 2.3]{Bruinier}.

\end{rem}

\begin{comment}
\section{Application to the Unitary case}
\label{apptounitary}

Now we construct the regularized theta lift on unitary groups.
Through the canonical injection $\U(V_E)\hookrightarrow\SO(V_F)$, we obtain the embedding of Shimura varieties $\iota\colon M_{K_f^G}\hookrightarrow N_{K_f^H}$.

\begin{lem}
\label{Lemma:res}
For $x\in V_E$ and $g\in G(\A_f)$,
\[\iota(Z^G(x,g)_{K_f^G})=\iota(M_{K_f^G})\cap Z^H(x,g)_{K_f^H}\]
\end{lem}
\begin{proof}
See \cite[Lemma 6.1]{Hofmann}
\end{proof}
\end{comment}

\begin{comment}
We get the regularized theta lift on unitary groups.
\begin{thm}
\label{unitarylift}
Let $f=\sum_{\mu}\sum_{m}c(m,\mu)f_{m,\mu}\in M!_{k,\overline{\rho_L}}$ be a weakly holomorphic Whittaker form of weight $k$ for $\Gamma$ whose coefficients $c(m,\mu)$ are integral.
Then there exists a meromorphic modular form $\Psi_f(\tau,g)$ for $G(\Q)$ of level $K_f^G$ satisfing
\begin{enumerate}
\item The weight of $\Psi$ is $-B(f)$.

\item $\div\Psi=\iota^{-1}(Z(f))=\Sigma_{\mu}\Sigma_{m} c(m,\mu)Z^G(m,\mu)_{K_f^G}$
\end{enumerate}
\end{thm}

\begin{proof}
Pull back the regularized theta lift on orthogonal groups constructed in Theorem \ref{orthoglift} and use Lemma
 \ref{Lemma:res}.
\end{proof}
This proves Theorem \ref{MainTheorem1}.
\end{comment}

\section{Modularity of special cycles on unitary groups for $e=1$ case}
\label{Section:Modularity}
\subsection{Divisors case}

\begin{thm}
Assume $e=1$ and $r=1$.
Then $Z_{\phi_f}^G(\tau)$ is an elliptic modular form of weight $n+1$.
\end{thm}

\begin{proof}
By \cite[Corollary 3.4]{Liu}, we have $\iota^{\star}Z_{\phi_f}^H(\tau)=Z_{\phi_f}^G(\tau)$.
  So by Remark \ref{rem:modularity on orthogonal groups}, the generating series $Z_{\phi_f}^G(\tau)$ is an elliptic modular form.
  Since the weight of $Z_{\phi_f}^H(\tau)$ is $n+1$, this finishes the proof.
\end{proof}
This gives a proof of Theorem \ref{MainTheorem2}.
Here we remark that to prove the modularity of $Z_{\phi_f}^G(\tau)$ for $n>1$, we use the perfect pairing presented in Proposition \ref{pairing}.
For $n=1$, we use an embedding trick. 
For more details, See \cite{Bruinier} or \cite{Yota}.

\subsection{General $r$ case}
To show the Hermitian modularity, we reduce the problem to the generators of the associated unitary group.
Now the indefinite unitary group $\U(r,r)$ is generated by the parabolic subgroup $P_r(F)=M_r(F)N_r(F)$ and $w_{r,r-1}$ where

\[M_r(F)\defeq\{m(a)=\begin{pmatrix}
a & 0 \\
0 & \overline{\transp{a}^{-1}}
\end{pmatrix} |\ a\in \GL_r(E) \}\]
\[N_r(F)\defeq\{n(u)=\begin{pmatrix}
1_r & u \\
0 & 1_r
\end{pmatrix} |\ u\in \Her_r(E)\}\]
\[w_{r,r-1}\defeq \begin{pmatrix}
1_{r-1} & 0 & 0_{r-1} & 0 \\
0 & 0 & 0& 1 \\
 0_{r-1} & 0 & 1_{r-1} & 0 &  \\
0    & -1 & 0 & 0 &
\end{pmatrix}\]
by \cite[Theorem 3.5]{Liu}.
We put $w_1\defeq w_{1,0}$.
By induction on $r$, we get the following result.
\begin{cor}
\label{cor:e=1,modularity}
Assume $e=1$.
Then $Z_{\phi_f}^G(\tau)$ is a Hermitian modular form of weight $n+1$ and genus $r$.
\end{cor}

\begin{proof}
  To prove that the generating series $Z_{\phi_f}^G(\tau)$ is a Hermitian modular form for $r=1$, we already know the modularity for $\SU(1,1)$.
  So, in particular, we know the modularity for the element $w_1\in\SU(1,1)$.
  Hence it suffices to prove the modularity for the parabolic subgroup $P_1\subset\U(1,1)$ because $\U(1,1)$ is generated by $P_1$ and $w_1=w_{1,0}$.
  We can prove the invariance under $P_1$ in the same way as \cite{Liu} or \cite{Yota}.
  This finishes the proof of the corollary for the $r=1$ case.
  For $r>1$, we use induction on $r$.
  More specifically, for any $r$, we can prove the modularity for $P_r$ in the same way as in subsection \ref{subsection:parabolic}, i.e.
  \begin{align*}\omega_f(n(u)_fg_f')(\phi_f\otimes\varphi_+^d)(x)Z^G(x)_{K_f^G}&=\omega_f(g_f')(\phi_f\otimes\varphi_+^d)(x)Z^G(x)_{K_f^G}\\
\omega_f(m(a)_fg_f')(\phi_f\otimes\varphi_+^d)(x)Z^G(x)_{K_f^G}&=\omega_f(g_f')(\phi_f\otimes\varphi_+^d)(x)Z^G(x)_{K_f^G}\end{align*}
hold for any $u\in \Her_r(F)$ and $a\in\GL_r(F)$.
  By using the modularity for $w_1$ in the $r=1$ case, we can prove the modularity for $w_{r,r-1}$ when $r>1$ in the same way as in subsection \ref{subsection:w_1} and we already know the $w_1$-modularity.
  For the step reducing the problem to the $r=1$ case by induction, see the proof of Theorem \ref{unitarymodular} in subsection \ref{subsection:w_1}.
\end{proof}
  This shows the modularity of special cycles on a unitary Shimura variety for $e=1$ (Theorem \ref{MainCor1}) and gives another proof of Liu's one.

\section{general $e$ case}
\subsection{Weil representaions}
Let $\psi\colon E\backslash\A_E\to \C^{\times}$ be the composite of the trace map  $E\backslash\A_E\to\Q\backslash\A$ and the usual additive character
\begin{align*}
\Q\backslash\A&\to\C^{\times}\\
(x_v)_v&\mapsto\exp(2\pi\sqrt{-1}(x_{\infty}-\sum_{v<\infty}\overline{x_v})),
\end{align*}
where $\overline{x_v}$ is the class of $x_v$ in $\Q_p/\Z_p$.

Let $(W,(\ ,\ ))$ be a Hermitian space of dimension $2r$ over $E$ whose sign is $(r,r)$ so that $\U(W)=\U(r,r)$.
Then we get a symplectic vector space $\W\defeq\Res_{E/F}(V_E\otimes_E W)$ with the skew-symmetric form $\Tr_{E/F}(\langle\ ,\ \rangle\otimes(\ ,\ ))$.
Let $\Sp(\W)$ be the symplectic group and $\Mp(\W)$ be its metaplectic $\C^{\times}$ covering group.
Then we get the Weil representaion $\omega_f$ and $\omega_{\A}$, the action of  $\Mp(\W)(\A_f)$ to $\S(V(\A_{F,f})^r)$ and $\Mp(\W)(\A)$ to $\S(V(\A_F)^r)$.

Now we state the second solution to Conjecture \ref{Conjecture}.
\begin{thm}
  \label{unitarymodular}
  $Z^G_{\phi_f}(\tau)$ is a Hermitian modular form of weight $n+1$ and genus $r$ under the Beilinson-Bloch conjecture with respect to orthogonal Shimura varieties for $m=e$.
\end{thm}

We reduce Theorem \ref{unitarymodular} to the orthogonal case, so have to assume the Beilinson-Bloch conjecture with respect to orthogonal Shimura varieties.
The strategy is as follows.
For the general $e$ case, we can prove the modularity for $P_r$ for any $r$ by direct calculation.
We can also show the modularity for $w_{r,r-1}$ when $r>1$, assuming the modularity for $w_1=w_{1,0}$ in the $r=1$ case.
So the problem is the modularity for $w_1$ for $r=1$ and general $e$.
We treat this problem by embedding unitary Shimura varieties into orthogonal varieties, studied in \cite{Hofmann}.
In orthogonal cases, the modularity of generating series are proved by \cite{Remarks} or \cite{Yota} under the Beilinson-Bloch conjecture.
We remark that when $e=1$, the modularity for $w_1$ is solved by Corollary \ref{cor:e=1,modularity}, followed by the modularity for $\SU(1,1)$ using the regularized theta lifts.
For the precise statement of the Beilinson-Bloch conjecture, see \cite[Subsection 1.2]{Yota}.

By \cite{YZZ}, we get the following expression for the generating series for the unitary group $G$.
\[Z^G_{\phi_f}(\tau)=\sum_{\substack{x\in K_f^G\backslash \widehat{V}_E^{r-1} \\ \mathrm{admissible}}}\sum_{y_2\in Fx}\sum_{\substack{y_1\in K_{f,x}^G\backslash x^{\perp} \\ \mathrm{admissible}}}\phi_f(x,y_i+y_2)Z^G(y_1)_{K_{f,x}}q^{T(x,y_1+y_2)}\]
where $K_{f,x}^G$ is the stabilizer of $x$, $\widehat{V}_E\defeq V_E\otimes \mathbb{A}_{f}$ and 
\[q^T(x)\defeq \exp(2\pi\sqrt{-1}\sum_{i=1}^{d}\Tr\tau_i \sigma_i T(x))\]
for $\tau=(\tau_1,\dots,\tau_d)\in(\H_r)^d$ and the moment matrix $T(x)=((x_i,x_j)/2)_{i,j}$.
Here for the notion  ``admissible" and $Z^G(x)_{K_f}$, see \cite[Lemma 3.1]{Liu}, \cite[Lemma 2.1]{Yota} or \cite[Lemma 2.1]{YZZ}.
Let $\varphi_+(x)=\exp(-\pi \Tr T(x))$ be the Gaussian.
Let $Z_{\phi_f}(g')$ be a function on $\U(r,r)(\A_F)$ defined by 
\begin{eqnarray*}
Z^G_{\phi_f}(g')
&\defeq&\sum_{x\in G(\Q)\backslash V_E^r}\sum_{g\in G_x(\A_f)\backslash G(\A_f)/K_f^G}\omega_{\A}(g')(\phi_f\otimes\varphi_+^d)(g^{-1}x)Z^G(x,g)_{K_f^G} \\
&=&\sum_{\substack{x\in K_f^G\backslash \widehat{V}_E^{r-1} \\ \mathrm{admissible}}}\sum_{y_2\in Fx}\sum_{\substack{y_1\in K_{f,x}^G\backslash x^{\perp} \\ \mathrm{admissible}}}\omega_{\A}(g')(\phi_f\otimes\varphi_+^d)(x,y_1+y_2)Z^G(y_1)_{K_{f,x}^G}.
\end{eqnarray*}

\begin{rem}
  The modularity of the generating series $Z_{\phi_f}(\tau)$ is equivalent to the left $\U(r,r)(F)$-invariance of the function $Z_{\phi_f}(g')$ on $\U(r,r)(\A)$.
\end{rem}

So in the following, we show
 the left  $\U(r,r)$-invariance of  $Z_{\phi_f}(g')$.
 First, we show the $P_r$-invariance  of  $Z^G_{\phi_f}(g')$ for any $r$.
Second, for $r>1$, we show the $w_{r,r-1}$-invariance of $Z_{\phi_f}(g')$, assuming the $w_1$-invariance for the  $r=1$ case.
Finally, we show that $Z^G_{\phi_f}(g')$ is $w_1$-invariant for the $r=1$ case.

\subsection{Invariance under the parabolic subgroup $P_r$}
\label{subsection:parabolic}
The elements $m(a)$ and $n(u)$ generate the parabolic subgroup $P_r(F)\subset \U(r,r)(F)$.

%For $g'\in\U(r,r)(\A_F)$, its infinity component in $\U(r,r)(F_{\infty})\cong\prod_{i=1}^d\U(r,r)(\R)$ is denoted by $g_{\infty}'=(g_{\infty,1}',\dots,g_{\infty,d}')$.

By the same way  \cite[Theorem 3.5 (1)]{Liu} or \cite[Section 4.1]{Yota}, we can show the  following equations with respect to $n(u)$ and $m(a)$:
\begin{align*}
    \omega_f(n(u)_fg_f')(\phi_f\otimes\varphi_+^d)(x)Z^G(x)_{K_f^G}&=\omega_f(g_f')(\phi_f\otimes\varphi_+^d)(x)Z^G(x)_{K_f^G}\\
    \omega_f(m(a)_fg_f')(\phi_f\otimes\varphi_+^d)(x)Z^G(x)_{K_f^G}&=\omega_f(g_f')(\phi_f\otimes\varphi_+^d)(xa)Z^G(x)_{K_f^G}
\end{align*}
for any $u\in \Her_r(F)$ and $a\in\GL_r(F)$.
First equation means the $n(u)$-invariance of $Z_{\phi_f}(g')$.
We shall prove that $Z_{\phi_f}(g')$ is $m(a)$-invariant as folows.
We have $U(x)=U(xa)$, so $Z^G_{\phi_f}(x)=Z^G_{\phi_f}(xa)$.
Therefore, combining the above calculation and the fact $Z^G_{\phi_f}(x)=Z^G_{\phi_f}(xa)$, we conclude that
\begin{eqnarray*}
Z^G_{\phi_f}(\omega_f(m(a))g')
&=& \sum_{\substack{x\in K_f^G\backslash \widehat{V}_E^r \\ \mathrm{admissible}}}\omega_f(g_f')(\phi_f\otimes\varphi_+^d)(xa)Z^G(xa)_{K_f^G} \\
&=& \sum_{\substack{x\in K_f^G\backslash \widehat{V}_E^r \\ \mathrm{admissible}}}\omega_f(g_f')(\phi_f\otimes\varphi_+^d)(x)Z^G(x)_{K_f^G} \\
&=& Z^G_{\phi_f}(g').
\end{eqnarray*}
This shows $Z_{\phi_f}(g')$ is invariant under the action of the  parabolic subgroup $P_r$.

\begin{comment}
$n(u)$ acts as follows:
\[\omega_f(n(u)_f)\phi_f(x)=\psi_f(\Tr(u_fT(x)))\phi_f(x).\]
Thus, we have
\begin{eqnarray*}
& &  \omega_f(n(u)_fg_f')(\phi_f)(x)Z(x)_{K_f}\prod_{i=1}^{d}W_{T(x)}(n(u)_{\infty,i}g_{\infty,i}') \\ &=&\psi_f(\Tr(u_fT(x_f)))\omega_f(g_f')\phi_f(x)Z(x)_{K_f}\psi_{\infty}(\sum_{i=1}^{d}\Tr(u_{i,\infty}T(x)_{\infty,i}))\prod_{i=1}^{d}W_{T(x)}(g_{\infty,i}') \\
&=&\psi(\Tr(uT(x)))\omega_f(g_f')\phi_f(x)Z(x)_{K_f}\prod_{i=1}^{d}W_{T(x)}(g_{\infty,i}') \\
&=&\omega_f(g_f')(\phi_f)(x)Z(x)_{K_f}\prod_{i=1}^{d}W_{T(x)}(g_{\infty,i}').
\end{eqnarray*}
\end{comment}

\subsection{Invariance under $w_{r,r-1}$ for $r>1$}
\label{subsection:w_1}
%Here, $W_{T(x_1+x_2,y)}(g_{\infty}')=\omega_{\infty}(g_{\infty}')(\varphi)$
Now, assuming $r=1$ case, we have
\begin{eqnarray*}
Z^G_{\phi_f}(w_{r,r-1}g')
&=&\sum_{\substack{x\in K_f^G\backslash \widehat{V}_E^{r-1} \\ \mathrm{admissible}}}\sum_{y_2\in Ex}\sum_{\substack{y_1\in K_{f,x}^G\backslash x^{\perp} \\ \mathrm{admissible}}}\omega_{\A}(g')(\phi_f\otimes\varphi_+^d)^{y_2}(x,y_1+y_2)Z^G(y_1)_{K_{f,x}^G}
\end{eqnarray*}
where $\phi^y(x,y)$ is the partial Fourier transformation with respect to the second coordinate.
Here we use the fact that \[\omega_{\A}(w_1)(\phi_f\otimes\varphi_+^d)(x,y)=(\phi_f\otimes\varphi_+^d)^{y}(x,y).\]
By the Poisson summation formula, this equals to
\[\sum_{\substack{x\in K_f^G\backslash \widehat{V}_E^{r-1} \\ \mathrm{admissible}}}\sum_{y_2\in Ey}\sum_{\substack{y_1\in K_{f,x}^G\backslash x^{\perp} \\ \mathrm{admissible}}}\omega_{\A}(g')(\phi_f\otimes\varphi_+^d)(x,y_1+y_2)Z^G(y_1)_{K_{f,x}^G},\]
which coincides with the definition of $Z^G_{\phi_f}(g')$.
Therefore, we get
\[Z^G_{\phi_f}(w_{r,r-1}g')=Z^G_{\phi_f}(g').\]
This shows that the function $Z_{\phi_f}(g')$ is invariant under the action of the element $w_{r,r-1}$.
\subsection{Invariance under $w_1$ for $r=1$}
\label{subsection:w_1,r=1}
We use Liu's proof \cite[Theorem 3.5]{Liu}.
Now $\U(1)\times\U(1)$ is a maximal compact subgroup of $\U(1,1)$ and $\SL_2(\A_{F,f})(\U(1)\times\U(1))(\A_{F,f})=\U(1,1)(\A_{F,f})$.
So we reduce the problem  to proving that $Z^G_{\phi_f}(w_1g')=Z^G_{\phi_f}(g')$ for all $g'\in\SL_2(\A_F)$.
By \cite[Corollary 3.4]{Liu} and the proof of \cite[Lemma 3.6]{Liu},
it suffices to prove $Z^H_{\phi_f}(w_1g')=Z^H_{\phi_f}(g')$.
However, this follows by \cite{Yota} under the Beilinson-Bloch conjecture with respect to orthogonal Shimura varieties.
This finishes the proof of Theorem \ref{unitarymodular}.

\subsection*{Acknowledgements}
The author would like to thank to his advisor, Tetsushi Ito, for useful discussions and warm encouragement.
I would also like to thank Jiacheng Xia for letting me know his paper \cite{Xia} and suggesting stating Corollary \ref{restate} explicitly.

\end{document}